\documentclass[11pt, oneside]{article}

\usepackage{amsmath, amssymb, amsthm}
\usepackage{enumitem}
\usepackage{capt-of}
\usepackage{ifthen}

\newtheorem{theorem}{Theorem}

\newtheorem{observation}[theorem]{Observation}

\long\def\symbolfootnote[#1]#2{\begingroup\def\thefootnote{\fnsymbol{footnote}}
\footnote[#1]{#2}\endgroup}

\begin{document}

\title{A note on forbidding clique immersions}
\author{
Matt DeVos\footnote{mdevos@sfu.ca. Supported in part by an NSERC Discovery Grant (Canada) and a Sloan Fellowship.}
\and
Jessica McDonald\footnote{jessica\textunderscore mcdonald@sfu.ca. Supported by an NSERC Postdoctoral Fellowship (Canada)}
\and
Bojan Mohar\footnote{mohar@sfu.ca. Supported in part by an NSERC Discovery Grant (Canada), by the Canada Research Chair program, and by the Research Grant J1--4106 of ARRS (Slovenia). On leave from: IMFM \& FMF, Department of Mathematics, University of Ljubljana, Ljubljana, Slovenia.}
\and
Diego Scheide\footnote{dscheide@sfu.ca}
\medskip\\
Department of Mathematics\\
Simon Fraser University\\
Burnaby, B.C., Canada V5A 1S6
}
\date{}

\maketitle

\begin{abstract}
Robertson and Seymour proved that the relation of graph immersion is well-quasi-ordered for finite graphs. Their proof uses the results of graph minors theory. Surprisingly, there is a very short proof of the corresponding rough structure theorem for graphs without $K_t$-immersions; it is based on the Gomory-Hu theorem. The same proof also works to establish a rough structure theorem for Eulerian digraphs without $\vec{K}_t$-immersions, where $\vec{K}_t$ denotes the bidirected complete digraph of order $t$.
\end{abstract}

\section{Introduction}

In this paper all graphs and digraphs are finite and may have loops and multiple edges, unless explicitly stated otherwise.

A pair of distinct adjacent edges $uv$ and $vw$ in a graph
are \emph{split off} from their common vertex $v$ by deleting the edges
$uv$ and $vw$, and adding the edge $uw$ (possibly in parallel to an existing edge,
and possibly forming a loop if $u = w$). A graph $H$ is
said to be \emph{immersed} in a graph $G$ if a graph isomorphic to
$H$ can be obtained from a subgraph of $G$ by splitting off pairs of
edges (and removing isolated vertices).
If $H$ is immersed in a graph $G$, then we also say
that $G$ has an \emph{$H$-immersion}.
An alternative definition is that $H$ is immersed in $G$ if there is
a 1-1 function $\phi: V(H)\to V(G)$ such that for each edge $uv\in E(H)$,
there is a path $P_{uv}$ in $G$ joining vertices $\phi(u)$ and $\phi(v)$,
and the paths $P_{uv}$, $uv\in E(H)$, are pairwise edge-disjoint.

Roberston and Seymour \cite{RS23} proved that the relation of graph immersion is a well-quasi-ordering, that is, for every infinite set of graphs, one of them can be immersed in another one. Their proof is based on a significant part of the graph minors project. It is perhaps surprising then that there is a very short proof of the corresponding rough structure theorem for graphs without $K_t$-immersions. Moreover, the same proof technique also works to prove a rough structure theorem for Eulerian digraphs without $\vec{K}_t$-immersions. Here, by $\vec{K}_t$ we mean the \emph{complete digraph} of order $t$, having $t$ vertices and a digon (pair of oppositely oriented edges) between each pair of vertices. By immersion in digraphs we mean the natural directed analogue of immersion in undirected graphs. That is, we say that a digraph $F$ is \emph{immersed\/} in a digraph $D$ if there is
a 1-1 function $\phi: V(F)\to V(D)$ such that for each edge $uv\in E(F)$,
there is a directed path $P_{uv}$ in $D$ from $\phi(u)$ to $\phi(v)$,
and the paths $P_{uv}$, $uv\in E(F)$, are pairwise edge-disjoint. Given a directed path $uvw$ of length two in a digraph, the pair of  edges $uv$, $vw$ are \emph{split off} from $v$ by deleting the edges $uv$ and $vw$, and adding the edge $uw$.

To state the two rough structure theorems explicitly, we first need the definition of a laminar family of edge-cuts. Given a graph or digraph $G$ and a vertex-set $X\subseteq V(G)$, we denote by $\delta(X)$ the edge-cut of $G$ consisting of all edges between $X$ and $V(G)\setminus X$ (in both directions). Two edge-cuts $\delta(X)$ and $\delta(Y)$ in a connected graph or digraph are \emph{uncrossed\/} if either $X$ or $V(G) \setminus X$ is contained in either $Y$ or $V(G) \setminus Y$.  Two edge-cuts in a general graph or digraph are \emph{uncrossed\/} if this holds in each component.  A family of pairwise uncrossed edge-cuts is called \emph{laminar}.

\begin{theorem}\label{undirstruct}
For every graph which does not contain an immersion of the complete graph $K_{t}$, there exists a laminar family of edge-cuts, each with size $< (t-1)^2$, so that every block of the resulting vertex partition has size less than $t$.
\end{theorem}

\begin{theorem}\label{dirstruct}
For every Eulerian digraph $D$ which does not contain an immersion of $\vec{K}_t$, there exists a laminar family of edge-cuts, each with size $< 2t(t -1)$, so that every block of the resulting partition has size less than $t$.
\end{theorem}

We will show that the Gomory-Hu Theorem (stated in Section 2) yields very easy proofs of the Theorems \ref{undirstruct} and \ref{dirstruct}. In the undirected case we originally had somewhat lesser bounds.  Theorem \ref{undirstruct} as stated above is due to Seymour and Wollan \cite{SeW}, whose work we learned of later. The discrepancy resulted from our use of an approximation of Kriesell's Conjecture \cite{Kr}, which can be replaced in our argument with the following observation of theirs.

\begin{observation}\label{SW} Let $H_t$ be the graph obtained from $K_{1,t-1}$ by replacing each edge with $t-1$ parallel edges. Then $H_t$ has a $K_{t}$-immersion.
\end{observation}

\begin{proof} Let $v_1$ be the vertex of degree $(t-1)^2$ and let $v_2, \ldots v_t$ be the vertices of degree $t-1$. Label the $t-1$ edges between $v_1$ and $v_j$ by $\{e_{j,1},e_{j,2}, \ldots e_{j, t} \} \setminus e_{j,j}$ for $2 \le j \le t$. Then we have an immersion of $K_t$ on the vertices $v_1, \ldots, v_{t}$, where the requisite paths between $v_1$ and $v_2, \ldots v_{t}$ are given by $e_{2,1}, \ldots e_{t,1}$, and for every pair of vertices $v_i, v_j$ with $1<i< j\leq t$, the path between $v_i$ and $v_j$ is given by $e_{ij}e_{ji}$.\end{proof}

The bounds of Theorems \ref{undirstruct} and \ref{dirstruct} provide rough structure to the same extent. Namely, while the laminar family of Theorem \ref{undirstruct}  does not prohibit  a $K_t$-immersion, it does indeed prohibit a $K_{t^2}$-immersion. Analogously, the laminar family of Theorem \ref{dirstruct} also prohibits a $\vec{K}_{t^2}$-immersion.

That we restrict ourselves to Eulerian digraphs in Theorem \ref{dirstruct} should not be surprising. First of all, let us observe that digraph immersion is not a well-quasi-order in general. Furthermore, the present authors have exhibited in \cite{DMMS} that there exist simple non-Eulerian digraphs with all vertices of arbitrarily high in- and outdegree which do not contain even a $\vec{K}_3$-immersion. (Here, by simple digraph we mean a digraph $D$ with no loops and  at most one edge from $x$ to $y$ for any $x,y \in V(D)$, but where digons are allowed).
On the positive side, it has been shown by Chudnovsky and Seymour \cite{ChSe} that digraph immersion is a well-quasi-order for tournaments. That Eulerian digraphs of maximum outdegree $k$ are well-quasi-ordered by immersion was proved (although not written down) by Thor Johnson as part of his PhD thesis \cite{Jo}.

In \cite{DDFMMS}, the present authors along with Fox and Dvo\v{r}\'{a}k, proved that (undirected) simple graphs with minimum degree at least $200t$ contain a $K_t$-immersion.
In another paper \cite{DMMS}, the following positive result is obtained for digraphs when the Eulerian condition is added.

\begin{theorem} \emph{\cite{DMMS}}
\label{mindeg}
Every simple Eulerian digraph with minimum outdegree at least $t(t-1)$ contains an immersion of $\vec{K}_t$.
\end{theorem}

Theorem \ref{mindeg} is in fact directly implied by Theorem \ref{dirstruct}. This provides an alternate proof of Theorem \ref{mindeg} to that given in \cite{DMMS}. To see this, suppose that a simple Eulerian digraph $D$ has minimum outdegree at least $t(t-1)$, and contains a set $S\subseteq V(D)$ of  $|S|=s$ vertices with $|\delta(S)| < 2t(t-1)$. Then
$$s\cdot t(t-1) \leq |E[S]|+\tfrac{1}{2}|\delta(S)| < 2s(s-1)+t(t-1).$$
This implies that $s>t(t-1)/2$. When $t>2$, this tells us that $s>t$, and hence Theorem \ref{dirstruct} guarantees that $D$ has a $\vec{K}_t$-immersion.

\section{Proofs}

An edge-cut $\delta(X)$ in a graph $G$ is said to \emph{separate} a pair of vertices $x,y\in V(G)$ if $x\in X$ and $y\in V(G)\setminus X$ (or vice versa). Given a tree $F$ and an edge $e\in E(F)$, there exists $X\subseteq V(G)$ such that $\delta(X)=e$; $\delta(X)$ is called a \emph{fundamental cut\/} in $F$ and is associated with the vertex partition $\{X, V(F)\setminus X\}$ of $V(F)$.

\begin{theorem}[Gomory-Hu \cite{GH}]
For every graph $G$, there exists a tree $F$ with vertex set $V(G)$ and a function $\mu : E(F) \rightarrow {\mathbb Z}$ so that the following hold.
\begin{itemize}
\item For every edge $e\in E(F)$ we have that $\mu(e)$ equals the size of the edge-cut of $G$ given by the vertex partition associated with the fundamental cut of $e$ in the tree $F$.
\item For every $u,v \in V(G)$ the size of the smallest edge-cut of $G$ separating $u$ and $v$ is the minimum of $\mu(e)$ over all edges $e$ on the path in $F$ from $u$ to $v$.
\end{itemize}
\end{theorem}

\bigskip

\noindent{\it Proof (Theorem \ref{undirstruct}):}
Let $G$ be an arbitrary graph, and apply the Gomory-Hu Theorem to choose a tree $F$ on $V(G)$ and an associated function $\mu$.  Let ${\mathcal C}$ be the family of edge-cuts of $G$ that are associated with edges $e \in E(F)$ for which $\mu(e) < (t-1)^2$. We show that if any block of the resulting vertex partition has size $\geq t$, than $G$ contains a $K_t$-immersion. To this end, suppose we have such a block with distinct vertices $ v_1, v_2 \ldots, v_t$. Then every edge-cut separating these $t$ vertices has size $\ge (t-1)^2$. Hence, by Menger's Theorem, there exist $(t-1)^2$ edge disjoint paths starting at $v_1$ so that exactly $t-1$ of them end at each $v_j$ for $2 \le j \le t$. (To see this, consider adding an auxiliary vertex $w$ to the graph, adding $t-1$ parallel edges between $w$ and each of $v_2, \ldots, v_{t}$, and then applying the standard version of Menger's Theorem to $v$ and $w$.)  Hence $G$ immerses the graph $H_t$ as in Observation \ref{SW}.\hfill$\Box$

\bigskip

The proof of Theorem \ref{dirstruct} requires two additional classical theorems.

\begin{theorem}[Mader \cite{Ma}]\label{mader}
If $D$ is an Eulerian digraph and $v \in V(D)$ has nonzero degree, then there exists a pair of edges that can be split off from $v$ so that for all $u,u' \in V(D) \setminus \{v\}$, the size of the smallest edge-cut separating $u$ and $u'$ does not change.
\end{theorem}

We say that a digraph $D$ is \emph{strongly $k$-edge-connected\/} if $D - S$ is strongly connected for every $S \subseteq E(D)$ with $|S| < k$.  For a digraph $D$ and a vertex $v \in V(D)$ we define an \emph{arborescence with root} $v$ to be a subdigraph $T$ of $D$ which is a spanning tree in the underlying graph of $D$ and has the property that all edges are directed ``away'' from $v$.

\begin{theorem}[Edmonds \cite{Ed1,Ed2}]\label{eds}
Let $D$ be a digraph and let $r_1, \ldots, r_k$ be vertices of $D$ (not necessarily distinct).  If $D$ is strongly $k$-edge-connected then there exist $k$ edge-disjoint arborescences, $F_1, \ldots, F_k$, so that $F_i$ is rooted at $r_i$ for $1 \le i \le k$.
\end{theorem}

\bigskip

\noindent{\it Proof (Theorem \ref{dirstruct}):} Let $D$ be an arbitrary Eulerian digraph, and apply the Gomory-Hu Theorem to the underlying graph to choose a tree $F$ on $V(D)$ and an associated function $\mu$. Let ${\mathcal C}$ be the family of edge-cuts of $D$ that are associated with those edges $e \in E(F)$ with
$\mu(e) < 2t(t - 1)$. We will show that if any block of the resulting vertex partition have size $\geq t$, then $D$ contains a $\vec{K}_t$-immersion. To this end, suppose we have such a block with distinct vertices $ v_1, \ldots, v_t$.  Repeatedly apply Theorem \ref{mader} to split all other vertices in the graph completely.  The resulting digraph has vertex set $\{v_1,\ldots,v_t\}$ and strong edge-connectivity at least $t(t - 1)$.  It then follows from Theorem \ref{eds} that we may choose $t(t-1)$ edge-disjoint arborescences so that exactly $t-1$ of them are rooted at each $v_i$ for $1 \le i \le t$.  Now, for every $i$, $1 \le i \le t$, we may choose from the $t-1$ arborescences rooted at $v_i$ a collection of $t-1$ edge-disjoint directed paths each starting at $v_i$ and ending at a distinct point in $\{v_1,\ldots,v_t \} \setminus \{v_i\}$ (using one arborescence per path).  All together these paths give the desired immersion of $\vec{K}_t$.\hfill $\Box$


\begin{thebibliography}{00}

\bibitem{ChSe} M. Chudnovsky and P. D. Seymour.  A well-quasi-order for tournaments, J. Combin. Theory, Ser. B 101 (2011), 47--53.

\bibitem{DMMS} M. DeVos, J. McDonald, B. Mohar and D. Scheide. Immersing complete digraphs.
    Europ. J. Combin., to appear.

\bibitem{DDFMMS} M.~DeVos, Z.~Dvo\v{r}\'{a}k, J.~Fox, J.~McDonald, B.~Mohar and D.~Scheide.
Minimum degree forcing complete graph immersion,
Submitted to {\it Combinatorica.} http://arxiv.org/pdf/1101.2630

\bibitem{Ed1} J. Edmonds.
Submodular functions, matroids, and certain polyhedra. Combinatorial Structures and their Applications (Proc. Calgary Internat. Conf., Calgary, Alberta., 1969)
pp. 69--87, Gordon and Breach, New York, 1970.

\bibitem{Ed2} J. Edmonds.
Some well-solved problems in combinatorial optimization. Combinatorial programming: methods and applications (Proc. NATO Advanced Study Inst., Versailles, 1974),
pp. 285--301. NATO Advanced Study Inst. Ser., Ser. C: Math. and Phys. Sci., Vol. 19, Reidel, Dordrecht, 1975.

\bibitem{GH} R. E. Gomory and T. C. Hu. Multi-terminal network flows, J. Soc. Indust. Appl. Math. 9(4) (1961), 551--570.

\bibitem{Jo} T. Johnson. Eulerian Digraph Immersion, Ph. D. thesis, Princeton University, 2002.

\bibitem{Kr} M. Kriesell. Edge-disjoint trees containing some given vertices in a graph, J. Combin.
Theory Ser. B 88 (2003), 53--65.

\bibitem{Ma} W. Mader.
Konstruktion aller $n$-fach kantenzusammenh\"angenden Digraphen,
Europ. J. Combin. 3 (1982), no. 1, 63--67.

\bibitem{RS23} N. Robertson and P. D. Seymour.
Graph Minors XXIII, Nash-Williams' immersion conjecture,
J. Combin. Theory, Ser. B 100 (2010), 181--205.

\bibitem{SeW} P. D. Seymour and P. Wollan. The structure of graphs not admitting a fixed immersion, preprint, 2012.

\end{thebibliography}
\end{document}